\documentclass[11pt,a4paper]{article} 
\usepackage[utf8]{inputenc} 
\usepackage[T1]{fontenc} 
 \usepackage{amsmath, amssymb, amsthm} 
 \usepackage{mathtools}
  \usepackage[a4paper,margin=2.5cm]{geometry} \usepackage[colorlinks=true,linkcolor=blue,citecolor=blue,urlcolor=blue]{hyperref}    \newtheorem{theorem}{Teorema}[section]  \newtheorem{lema}[theorem]{Lema}  \theoremstyle{definition}   \theoremstyle{remark} \def\R{\mathbb R} \title{On the Minkowski deficit for surfaces of
revolution} \author{Joaquim Bruna, Julià Cufí and Agustí Reventós} \date{10 September,  2026} \begin{document} \maketitle \begin{abstract} By studying the classical Minkowski's inequality, $4\pi A\leq M^{2}$, relating area and mean curvature, in the particular case of convex surfaces of revolution, one is led to a functional inequality that holds for general  functions. We state and prove this inequality obtaining a new proof of Minkowski's inequality in this particular case and we establish a sharp inequality between the Minkowski deficit of the surface and the isoperimetric deficit of any meridian curve. \end{abstract}

\section{Introduction} \label{02a}
In this note we interpretate a classical Minkoswki's  
inequality for the special case of surfaces of revolution as a Wirtinger's type inequality with respect to a weight.

Denoting by $M$ the total mean curvature of a convex surface $S$ in $\R^3$ and by $F$
its area,  Minkowski's inequality states that $4\pi F\leq M^2$ and equality holds if and only if $S$ is a sphere.

On another hand if $p(t)$ is  a real function of class  ${\cal C}^2$ on the interval $[-\pi/2,\pi/2]$, satisfying 
\begin{equation*}p(t)>0,  \, p'(\pi/2)=p'(-\pi/2)=0,\,
p(t)+p''(t)>0,
\end{equation*}
the curve in the $x,z$ plane with support function $p(t)$
generates, by rotating about the $z$ axis, a differentiable convex surface $S$. It comes out that Minkowski's inequality for this surface of revolution $S$
 is equivalent to the inequality
 \begin{eqnarray*}\int_{-\pi/2}^{\pi/2}(2p^{2}-p'\,^{2})\cos(t)\,dt\leq \left(\int_{-\pi/2}^{\pi/2}p\cos(t)\,dt\right)^{2}.\end{eqnarray*}
 This is the content 
 of Section \ref{2804e}. 
In Theorem \ref{1405a} of section \ref{1405} we prove that the above integral inequality holds for any ${\cal C}^1$ function on $[-\pi/2,\pi/2]$ and we characterize when equality holds. 

So, in particular,   we provide a new proof  of Minkowski's inequality for the special case of surfaces of revolution. 

In section 4 we prove a sharp inequality between the Minkowski deficit of the surface $S$ and the isoperimetric deficit of any meridian curve $C$ of $S$. More concretely we show that
                                   $$M^2 - 4\pi F \leq  \pi (L^2 - 4\pi A),$$
where $L$ is  the length of $C$ and $A$ the area enclosed by $C$. The constant $\pi$ is sharp and equality holds only when $S$ is a sphere.
 
\section{Surfaces of revolution}\label{2804e}
We recall that  Minkoswki's  
inequality for a convex surface of area $A$ and total mean curvature $M$, states that 
\begin{equation}\label{1506}4\pi F\leq M^{2}\end{equation}
with equality only for spheres  \cite{Mink}.

In this section we consider this inequality for the special case of surfaces of revolution and relate it to functional integral inequality.

Let $p(t)$ be a real function of class  ${\cal C}^2$ defined in the interval $[-\pi/2,\pi/2]$, satisfying the conditions 
\begin{equation}\label{2804d}p(t)>0, \, p'(\pi/2)=p'(-\pi/2)=0,\,
p(t)+p''(t)>0.\end{equation}

Then the curve $C$ in the $x,z$ plane, with support function $p(t)$
\begin{eqnarray*}
x(t)&=&p(t)\cos(t)-p'(t)\sin(t)\\
z(t)&=&p(t)\sin(t)+p'(t)\cos(t)
\end{eqnarray*}
generates,   by rotating  about the $z$-axis,  a differentiable convex surface $S$.

The standard parametrization of $S$ is
$$X(t, \theta)=(x(t)\cos\theta, x(t)\sin\theta, z(t)), \quad -\pi/2\leq t\leq \pi/2, \; 	\;0\leq \theta\leq 2\pi.$$

The principal curvatures of $S$ are given by
$$k_{1}(t, \theta)=k(t), \quad k_{2}(t, \theta)=\frac{\cos(t)}{x(t)}$$
where $k(t)$ is the curvature of $C$.

Recall that the curvature radius  $\rho(t)$ of $C$
is given by
$$\rho(t)=\frac{1}{k(t)}=p(t)+p''(t).$$
The coefficients 
of the first and the second fundamental forms of $S$, with respect to the above parametrization, 
 are 
$$E=\rho(t)^{2},\; F=0, \; G=x(t)^{2}, \qquad e=\rho(t),\;  f=0,\;  g=x(t)\cos(t).$$
From this it follows that the mean curvature $H$ is given by
$$2H=k_{1}+k_{2}=\frac{Eg-2Ff+Ge}{EG-F^2}=\frac{\rho(t)\cos(t)+x(t)}{\rho(t)x(t)}$$
and the surface element is 
$$dS=|\rho(t)x(t)|dt\, d\theta=\rho(t)x(t)dt\, d\theta, $$
since $\rho(t)>0$ by hypothesis and it can be easily seen,  using  $x'(t)=-\rho(t)\sin(t)$,  that $x(t)>0$.

Hence, denoting by $M=\int_{S}HdS$ the total mean curvature, we have 
$$M=\pi\int_{-\pi/2}^{\pi/2}(2p(t)\cos(t)+p''(t)\cos(t)-p'(t)\sin(t)) dt$$
and taking into account that 
$$(p'(t)\cos(t))'=p''(t)\cos(t)-p'(t)\sin(t)$$
one has
\begin{equation}\label{2804}M=2\pi\int_{-\pi/2}^{\pi/2}p(t)\cos(t)\, dt.\end{equation}

Concerning the area of the surface $S$
we have 
\begin{eqnarray}\label{2704b}F&=&\int_{\partial K}\rho(t) x(t)dt d\theta=2\pi\int_{-\pi/2}^{\pi/2}(p+p'')(p\cos(t)-p'\sin(t))dt\nonumber
\\&=&2\pi\int_{-\pi/2}^{\pi/2}\big(p^{2}\cos(t)-pp'\sin(t)+pp''\cos(t)-p'p''\sin(t)\big)\,dt.
\end{eqnarray}
Integrating between $-\pi/2$ and $\pi/2$
the equalities 
\begin{eqnarray*}(pp'\cos(t))'&=&p'^2\cos(t)+pp''\cos(t)-pp'\sin(t),\\(p'^2\sin(t))'&=&2p'p''\sin(t)+p'^2\cos(t)
,\end{eqnarray*}
and taking into account that $p'(-\pi/2)=p'(\pi/2)=0$
we obtain
\begin{eqnarray*}
\int_{-\pi/2}^{\pi/2}p'^{2}\cos(t)\,dt&=&\int_{-\pi/2}^{\pi/2}pp'\sin(t)\,dt-\int_{-\pi/2}^{\pi/2}pp''\cos(t)\,dt,\\
\int_{-\pi/2}^{\pi/2}p'^{2}\cos(t)\,dt&=&-2\int_{-\pi/2}^{\pi/2}p'p''\sin(t)\,dt.
\end{eqnarray*}
Substituting  these values in \eqref{2704b}
we obtain

\begin{equation}\label{1404}F=2\pi\int_{-\pi/2}^{\pi/2} (p^2\cos(t)-\frac{1}{2}p'^2\cos(t))dt=\pi\int_{-\pi/2}^{\pi/2} (2p^2-p'^2)\cos(t) \,dt.\end{equation}


 Then, by \eqref{2804} and \eqref{1404}, the Minkowski's inequality $4\pi A\leq M^{2}$, in the particular case of a surface of revolution, is equivalent  to
\begin{eqnarray*}\int_{-\pi/2}^{\pi/2}(2p^{2}-p'\,^{2})\cos(t)\,dt\leq \left(\int_{-\pi/2}^{\pi/2}p\cos(t)\,dt\right)^{2}\end{eqnarray*}
for a function $p(t)$ satisfying the conditions  \eqref{2804d}.
For a such function $p(t)$ equality holds only when $p(t)$ is a constant, that is, the surface is a sphere.

\bigskip
We shall prove in the next section that the above inequality holds for a ${\cal C}^1$ function without any restriction.

\section{The Wirtinger type inequality}\label{1405}

The Wirtinger  inequality states that for a $2\pi$-periodic function $f(t)$ of class ${\cal C}^{1}$ with $\int_{0}^{2\pi}f(t)\,dt=0$ one has
$$\int_{0}^{2\pi}f^{2}(t)\,dt\leq \int_{0}^{2\pi}f'^{2}(t)\,dt,$$
with equality if and only if $f(t) = a \sin(t) + b \cos(t)$ for some $a,b\in\R$. 

\bigskip
We shall prove here an analogous  inequality when integration is taken with respect to a more general measure.

\begin{theorem}\label{1405a}
Let $p(t)$ be a real function of class ${\cal C}^1$ in $[-\pi/2,\pi/2]$. Then  
\begin{eqnarray}\label{feb23}\int_{-\pi/2}^{\pi/2}(2p^{2}-p'\,^{2})\cos(t)\,dt\leq \left(\int_{-\pi/2}^{\pi/2}p\cos(t)\,dt\right)^{2}.\end{eqnarray}
Equality holds if and only if  $p(t)=a_{0}+a_{1}\sin(t)$ with $a_{0},a_{1}$ some constants. 
\end{theorem}

Performing the change of variable $x=\sin(t)$ and denoting $u(x)=p(\arcsin(x))$
inequality \eqref{feb23} becomes
\begin{equation}\label{1006f}\int_{-1}^{1}\big(2u(x)^{2}-u'(x)^{2}(1-x^{2})\big)\,dx\leq \bigg(\int_{-1}^{1}u(x)\,dx\bigg)^{2},\end{equation}
where the function $u$ is continuous on $[-1,1]$, of class ${\cal C}^1$ on $(-1,1)$ and $u'(x)^2(1-x^2)$ is integrable on $(-1,1)$.

\medskip

To prove the above inequality we shall use the following 
\begin{lema}\label{lema1} Let $u$ be a real-valued function of class ${\cal C}^{2}$  on an open neighborhood of $[-1,1]$. Then
\begin{equation*}\label{1006}\int_{-1}^{1}\big(2u(x)^{2}-u'(x)^{2}(1-x^{2})\big)\,dx\leq \bigg(\int_{-1}^{1}u(x)\,dx\bigg)^{2}.\end{equation*}
Equality holds for $u(x)=a_{0}+a_{1}x$, with $a_{0},a_{1}$
some constants.\end{lema}
\begin{proof} 
Let 
$$u(x)=\sum_{n=0}^{\infty}a_{n}e_{n}(x)$$
be the development of the function  $u$ in the orthonormal  basis given by the normalized  Legendre polynomials   $e_{n}(x)=\dfrac{P_{n}(x)}{\|P_{n}\|}$, in the Hilbert space   ${\cal L}^2([-1,1])$ with the standard scalar product.

Note that $$\int_{-1}^{1}u(x)\,dx=\int_{-1}^{1}a_{0}\frac{1}{\sqrt{2}}\,dx=a_{0}\sqrt{2},$$ 
since $\int_{-1}^1 e_{n(x)}\,dx=0$, for $n\geq 1$.

\medskip
Also, by  Parseval's identity, it is  $$\int_{-1}^{1}u(x)^{2}\,dx=\sum_{n=0}^{\infty}a_{n}^{2}.$$

On the other hand, we recall that the Legendre differential operator 
is given by $$L[u](x)=-(u'(x)(1-x^{2}))'=-u''(x)(1-x^{2})+2xu'(x),$$
and, since $u(x)$ is of class ${\cal C}^{2}([-1,1])$, we have 
$L[u]\in {\cal L}^2([-1,1])$. Thus, one can write
$$L[u](x)=\sum_{n=0}^{\infty}d_{n}e_{n}(x).$$
We recall that $L$
is self-adjoint, and that the Legendre polynomials $P_{n}(x)$ are eigen-vectors of $L$ with eigen-values $n(n+1)$, \cite{Szego}. So we have
$$d_{n}=\langle L[u], e_{n}\rangle=\langle u, Le_{n}\rangle=\langle u, n(n+1)e_{n}\rangle= n(n+1)a_{n}.$$

Now,  using integration by parts and Plancherel's identity it follows 
\begin{eqnarray*}
\int_{-1}^{1}u'^{2}(1-x^{2})\,dx&=&\int_{1}^{1}u(-u'(1-x^{2}))'\,dx=\langle u, L[u]\rangle =\sum_{n=0}^{\infty}a_{n}d_{n}\\&=&
\sum_{n=0}^{\infty}n(n+1)a_{n}^{2}.
\end{eqnarray*}

Finally we have

\begin{eqnarray*}
\int_{1}^{1}2u^{2}\,dx-\int_{-1}^{1}u'^{2}(1-x^{2})\,dx&=&2\sum_{n=0}^{\infty}a_{n}^{2}-\sum_{n=0}^{\infty}n(n+1)a_{n}^{2}\\&=&
\sum_{n=0}^{\infty}a_{n}^{2}(2-n(n+1))\leq 2a_{0}^{2}=(\int_{-1}^{1}u(x)\,dx)^{2}.
\end{eqnarray*}
Equality hods when  $a_{k}=0$, for $k\geq 2$, that is $u(x)=a_{0}+a_{1}e_{1}(x)$. \hfill $\square$

\end{proof}

\medskip

{\em Proof of the Theorem  \ref{1405a}}. 
We shall approximate the function $u$ appearing in inequality \eqref{1006f} by means of functions satisfying the hypothesis of  Lemma \ref{lema1}.

For $0<r<1$ consider the functions $u_{r}(x)=u(rx)$. Since $u$ is of class ${\cal C}^1$ on $(-1,1)$, so is $u_{r}$ on $(-1/r,1/r)$ and we can extend $u_{r}$ to a ${\cal C}^{1}$ 
function with compact suport on $\R$, that still will be named $u_{r}$.

Let $\varphi_{n}$ be an approximation of the identity, and consider the convolutions $$(u_{r}\star \varphi_{n})(x)=\int_{\R}u_{r}(t)\varphi_{n}(x-t)\,dt.$$
These functions are ${\cal C}^{\infty}$with compact support  on $\R$ (see \cite{Lang}), and so, by Lemma \ref{lema1},  they satisfy inequality \eqref{1006f}.
Now, for $n\to \infty $, $u_{r}\star \varphi_{n}$ tends to $u_{r}$ and  $(u_{r}\star \varphi_{n})'$  tends to $u_{r}'$, uniformly. So, inequality \eqref{1006f} holds for $u_{r}$, that is
\begin{equation*}\int_{-1}^{1}\big(2u_{r}(x)^{2}-u_{r}'(x)^{2}(1-x^{2})\big)\,dx\leq \bigg(\int_{-1}^{1}u_{r}(x)\,dx\bigg)^{2}.\end{equation*}

Finally we need taking the limit as  $r\to 1$ in the above inequality. For this 
we will apply  the Lebesgue dominated convergence theorem. Since $u_{r}(x)$ is bounded independently of $r$,  we have
$$\lim_{r\to 1}\int_{-1}^{1}u_{r}(x)^2\,dx=\int_{-1}^{1}u(x)^2\,dx, \quad \lim_{r\to 1}\int_{-1}^{1}u_{r}(x)\,dx=\int_{-1}^{1}u(x)\,dx, $$

For the third integral    observe that
putting $y=rx$,
$$\int_{-1}^1 r^2u'(rx)^2(1-x^2)\,dx=\int_{-r}^r ru'(y)^2(1-\frac{y^2}{r^2})\,dy\leq \int_{-1}^1 u'(y)^2(1-y^2)\,dy,$$
and this last integral is finite by hypothesis. 

So
$$\lim_{r\to 1}\int_{-1}^1u_{r}'(x)^{2}(1-x^{2})\,dx=\int_{-1}^1u'(x)^{2}(1-x^{2})\,dx,$$
and inequality \eqref{1006f} is proved and so inequality \eqref{feb23} of Theorem \ref{1405a}.

If equality holds in \eqref{feb23}, equality also holds in Lemma \ref{lema1}, and so $u(x)=a_{0}+a_{1}x$
that implies $p(t)=a_{0}+a_{}\sin(t)$.  \hfill $\square$

\bigskip 

As said at the end of section \ref{2804e}, Theorem \ref{1405a} has as a consequence Minkowski's inequality \eqref{1506} for the special case of convex surfaces of revolution, showing also that the equality holds only for spheres.

\section{A sharp  upper bound for the Minkowski Deficit}
Consider a  surface  of revolution  $S$ obtained by  rotating  a curve  $C$ in the  $x,z$-plane about the  $z$-axis.
As in section \ref{02a}, suppose that 
$C$ is given by a support function $p(t)$, with  $-\pi/2\leq t\leq \pi/2$,  satisfying  $p(t)>0$ i $p'(\pi/2)=p'(-\pi/2)=0$.

The Minkowski deficit  of the surface   $S$
is the quantity
$$D_{M}(S)=M^{2}-4\pi F$$
where  $M$ is the total mean curvature of  $S$ and $F$ its area.
As we have seen in section \ref{2804e}, 
in terms of the support function $p(t)$
\begin{equation}\label{02b}D_{M}(S)=4\pi^{2}Q_{M}(p)\end{equation}
where
$$Q_{M}(p)=\int_{-\pi/2}^{\pi/2}p'^{2}\cos(t)\,dt-2\int_{-\pi/2}^{\pi/2}p^{2}\cos(t)\,dt+(\int_{-\pi/2}^{\pi/2}p\cos(t)\,dt)^{2}.$$

On the other hand,  the 
isoperimetric deficit of any meridian curve, for instance $C$ together with its reflection with  respect to the $z$-axis, is given by $$D_{iso}(C)=L^{2}-4\pi A$$
where $L$ is the length of the meridian curve  and $A$ is the area enclosed by it.
In terms of the support function, we have  \begin{equation}\label{02c}D_{iso}(C)=4\pi Q_{iso}(p)\end{equation}
where $$Q_{iso}(p)=\int_{-\pi/2}^{\pi/2}p'^{2}\,dt-\int_{-\pi/2}^{\pi/2}p^{2}\,dt+\frac{1}{\pi}(\int_{-\pi/2}^{\pi/2}p\,dt)^{2}.$$

%
\begin{theorem}\label{0909}The Minkowski deficit of a surface of revolution 
is bounded 
above 
by  $\pi$ times the  isoperimetric deficit of any  meridian curve. More precisely, with the notation above,  \begin{equation}\label{0909b}D_{M}(S)\leq \pi D_{iso}(C).\end{equation}
The constant $\pi$ is sharp and equality holds only when $S$ is a sphere.
\end{theorem}
\begin{proof} In view of the equalities  \eqref{02b} and \eqref{02c}, it suffices to prove  that
$$Q_{M}(p)\leq Q_{iso}(p).$$

Note that
\begin{eqnarray*}Q_{iso}(p)-Q_{M}(p)&=&\int_{-\pi/2}^{\pi/2}(1-\cos(t))p'^{2}\,dt+
\int_{-\pi/2}^{\pi/2}(2\cos(t)-1)p^{2}\,dt\\&+&\frac{1}{\pi}\big(\int_{-\pi/2}^{\pi/2}p\,dt\big)^{2}-\big(\int_{-\pi/2}^{\pi/2}p\cos(t)\,dt\big)^{2}.
\end{eqnarray*}
Thus, we need to prove  \begin{equation}\label{0509a}Q_{iso}(p)-Q_{M}(p)\geq 0.\end{equation}

Write $p=p_{e}+p_{o}$ for the decomposition of $p$ into its even and odd parts. 
We first  prove \eqref{0509a} for the even and odd parts separately.

\begin{lema}\label{0509}
For every smooth function $f$ on $[0,\pi/2]$,
define 
$${\cal L}(f)=\int_{0}^{\pi/2}\big((1-	\cos(t))f'^{2}+(2\cos(t)-1)f^{2}\big)\,dt.$$
 Then
$${\cal L}(f)=\int_{0}^{\pi/2}(1-\cos(t))\sin^{2}(t)\bigg[\big( \frac{f(t)}{\sin(t)}\big)' \bigg]^{2}\,dt.$$
In particular \(\mathcal L(f)\ge0\), and equality holds precisely when
\(f(t)=b\sin t\), for some constant $b$.
\end{lema}
\begin{proof}
The result follows directly from the identity 
$$(1-	\cos(t))f'^{2}+(2\cos(t)-1)f^{2}=(1-\cos(t))\sin^{2}(t)\bigg[\big( \frac{f(t)}{\sin(t)}\big)' \bigg]^{2}+\bigg((1-\cos(t))\frac{\cos(t)}{\sin(t)}f^{2}\bigg)'$$
which is easily verified. $\square$
\end{proof}

Fir the odd part, set $f=p_{o}\big|_{[0, \pi/2]}$. Since $p_{o}$ is odd, 
$$\int_{-\pi/2}^{\pi/2}p_{o}\,dt=\int_{-\pi/2}^{\pi/2}p_{o}\cos(t)\,dt=0.$$
Thus, by Lemma \ref{0509}, 
\begin{equation}\label{0909d}Q_{iso}(p_{o})-Q_{M}(p_{o})=2{\cal L}(f)\geq 0.\end{equation}

For the even part, let $h=p_{e}\big|_{[0,\pi/2]}$ and abbreviate  $$I(h)=\int_{0}^{\pi/2}h(t)\,dt,\qquad J(h)=\int_{0}^{\pi/2}h\cos(t)\,dt.$$
Then 
\begin{equation}\label{0509b}Q_{iso}(p_{e})-Q_{M}(p_{e})=2\bigg({\cal L}(h)+\frac{2}{\pi}I(h)^{2}-2J(h)^{2}\bigg).\end{equation}

Define the symmetric bilinear form for smooth functions on $[0,\pi/2]$ by 
\begin{equation*}
 \langle f,g\rangle
 :=\int_0^{\pi/2}
 \Bigl((1-\cos t)f'g'+(2\cos t-1)fg\Bigr)\,dt
 +\frac1a\left(\int_0^{\pi/2} f(t)\,dt\right)\left(\int_0^{\pi/2} g(t)\,dt\right).
\end{equation*}
It is positive definite.  Indeed, 
since $$ \langle f,f\rangle={\cal L}(f)+\frac{2}{\pi}\big(\int_{0}^{\pi/2}f(t)\,dt\big)^{2}$$
by Lemma~\ref{0509}, $ \langle f,f\rangle\geq 0$. Moreover $\langle f,f\rangle=0$ if and only if $f=0;$ hence the form is non-degenerate.

We then have 
\begin{equation*}
 \langle h,h\rangle
 =\mathcal L(f)+\frac{2}{\pi}I(h)^2.
\end{equation*}
Moreover,
\begin{align*}
 \langle h,1\rangle
 &=\int_0^{\pi/2}(2\cos t-1)h(t)\,dt+I(h)
 =2J(h),
\\
 \langle1,1\rangle
 &=\int_0^{\pi/2}(2\cos t-1)\,dt+\pi/2=2.
\end{align*}
The Cauchy--Schwarz inequality in \(\langle\cdot,\cdot\rangle\) therefore gives
\[
 4J(h)^2
 =\langle h,1\rangle^2
 \le \langle h,h\rangle\langle1,1\rangle
 =2\langle h,h\rangle.
\]
Equivalently,
\begin{equation*}
 2J(h)^2
 \le \mathcal L(h)+\frac{2}{\pi}I(h)^2.
\end{equation*}
Substitution into \eqref{0509b} proves
\begin{equation*}
 (Q_{iso}(p_{e})-Q_{M}(p_{e})) \geq 0.
\end{equation*}
Equality in Cauchy--Schwarz occurs exactly when \(h\) is a constant.  Hence the equality for the  even part holds for a  constant function.

\medskip
Finally, since the cross terms between the even and odd parts vanish, we have 
\begin{eqnarray}\label{0909c}Q_{iso}(p_{e}+p_{o})-Q_{M}(p_{e}+p_{o})&=&Q_{iso}(p_{e})-Q_{M}(p_{e})+Q_{iso}(p_{0})-Q_{M}(p_{o})\\&+& 		\int_{-\pi/2}^{\pi/2}(1-\cos(t))2p'_{e}p'_{o}\,dt+\int_{-\pi/2}^{\pi/2}2p_{e}p_{o}\,dt\geq 0,\nonumber\end{eqnarray}
because the two last integrands are odd functions.

\bigskip
It remains to show that the coefficient 1 in $Q_{M}(p)\leq Q_{iso}(p)$ cannot be decreased. Let $\eta $ be a smooth function on $\R$ with compact support contained in $(-1,1)$. 
For sufficiently small $\epsilon >0$, define

$$q_{\epsilon}(t)=\eta(\frac{t}{\epsilon}),\qquad -\pi/2\leq t\leq \pi/2.$$

The support of $q_{\epsilon}(t)$ is contained in $(-\epsilon, \epsilon)$. 
For $\epsilon<\pi/2, $ one has

\begin{eqnarray*}
\int_{-\pi/2}^{\pi/2}q_{\epsilon}(t)\,dt&=&\epsilon \int_{-\pi/2\epsilon}^{\pi/2\epsilon}\eta(u)\,du= \epsilon \int_{-1}^{1}\eta(u)\,du=\epsilon A,\\
\end{eqnarray*}
and analogously 
\begin{eqnarray*}
\int_{-\pi/2}^{\pi/2}q_{\epsilon}^{2}(t)\,dt&=&\epsilon \int_{-1}^{1}\eta^{2}(u)\,du=\epsilon B,\\
\int_{-\pi/2}^{\pi/2}[q_{\epsilon}'(t)]^{2}\,dt&=&\frac{1}{\epsilon}\int_{-1}^{1}[\eta'(u)]^{2}\,du=\frac{1}{\epsilon}C,\end{eqnarray*}
with $A,B,C$ 
 some constants.
 
Hence
$$Q_{iso}(q_{\epsilon})=\frac{1}{\epsilon}C-\epsilon B+\frac{1}{\pi}\epsilon^{2}A^{2}=\frac{1}{\epsilon}C+O(\epsilon).$$

We also have
\begin{eqnarray*}
\int_{-\pi/2}^{\pi/2}q_{\epsilon}(t)\cos(t)\,dt&=&\int_{-\pi/2}^{\pi/2}\eta(\frac{t}{\epsilon})\cos(t)\,dt=\epsilon\int_{-\pi/2\epsilon}^{\pi/2\epsilon}\eta(u)\cos(u\epsilon)du=\epsilon A+O(\epsilon^{3}),\\
\int_{-\pi/2}^{\pi/2}q_{\epsilon}^{2}(t)\cos(t)\,dt&=&\epsilon B+O(\epsilon^{3}),\\
\int_{-\pi/2}^{\pi/2}[q_{\epsilon}'(t)]^{2}\cos(t)\,dt&=&\frac{1}{\epsilon}C+O(\epsilon).\end{eqnarray*}

 Hence
 $$Q_{M}(q_{\epsilon})=\frac{1}{\epsilon}C+O(\epsilon).$$
 Thus 

$$\lim_{\epsilon\to 0}\frac{Q_{M}(q_{\epsilon})}{Q_{iso}(q_{\epsilon})}=1,$$
and the coefficient 1 in $Q_{M}(p)\leq Q_{iso}(p)$ cannot be decreased.

\medskip
We must still ensure that the test functions arise from strictly convex surfaces
of revolution. Set
$$p_{\epsilon}(t)=R_{\epsilon}+q_{\epsilon}(t), \qquad -\pi/2\leq t\leq \pi/2.$$
where $R_{\epsilon}>0$ is chosen sufficiently large so that $p_{\epsilon}(t)+p''_{\epsilon}(t)>0$.

Then, the surface $\Sigma_{\epsilon} $obtained by rotating de curve defined by the support function $p_{\epsilon}(t)$
is strictly convex and

$$\lim_{\epsilon\to 0}\frac{D_{M}(\Sigma_{\epsilon})}{D_{iso}(\Gamma_{\epsilon})}=\lim_{\epsilon \to 0}\pi\frac{Q_{M}(p_{\epsilon})}{Q_{iso}(p_{\epsilon})}=\pi, $$
where $\Gamma_{\epsilon}$ is a meridian curve of $\Sigma_{\epsilon}$, and this proves the optimality of $\pi$. 

 \medskip
Note that, geometrically, $\Sigma_{\epsilon}$ is obtained making  a small perturbation around the equator of a sphere. 
 
\medskip
Finally, if equality holds in \eqref{0909b} then also equality holds in \eqref{0909c} and so $Q_{iso}(p_{e})-Q_{M}(p_{e})=0$ and $Q_{iso}(p_{0})-Q_{M}(p_{o})=0$. As we have said, equality for the even part implies $p_{e}=a$ for some constant $a$, and  by \eqref{0909d} and Lemma \ref{0509}, it is  $p_{o}=b\sin(t)$. Therefore  $p(t)=a+b\sin(t)$ and $C$ is a circle and $S $
a sphere. $\square$
 \end{proof}


\medskip

{\em Acknowledgements}. During the preparation of this manuscript, the authors used AI (LLM got-5.6-sol) as an auxiliary tool in the proof of Lemma \ref{0509} and with the choice of the  symmetric bilinear form  introduced in the proof of Theorem \ref{0909}. The authors independently verified all mathematical statements and are solely responsible for the contents of the manuscript.

 \bibliographystyle{plain}
\bibliography{bibliografia}

\bigskip

{\em Joaquim Bruna.} Departament de Matemàtiques, Universitat Aut\`{o}noma de Barcelona\\ 08193 Bellaterra, Barcelona, Catalonia, joaquim.bruna@uab.cat

{\em Julià Cufí.} Departament de Matemàtiques, Universitat Aut\`{o}noma de Barcelona\\ 08193 Bellaterra, Barcelona, Catalonia, julia.cufi@uab.cat

{\em Agustí Reventós.} Departament de Matemàtiques, Universitat Aut\`{o}noma de Barcelona\\ 08193 Bellaterra, Barcelona, Catalonia, agusti.reventos@uab.cat
\end{document}